\documentclass{amsart}
\usepackage{amsmath}
\usepackage{amscd}
\usepackage{amssymb}
\usepackage[style=alphabetic]{biblatex}
 \newtheorem{thm}{Theorem}
 \newtheorem{prop}[thm]{Proposition}
 
 \newtheorem{lem}[thm]{Lemma}
 
 \theoremstyle{definition}
 
 \newtheorem{defin}{Definition}

\newcommand{\Aut}{\text{Aut}}

\begin{document}

\title[Normality of 3K]%
{Normality of the Kimura 3-parameter model }

\author[M.~Vodi\v{c}ka]%
{Martin Vodi\v{c}ka}

\address{%
{\sl Martin Vodi\v{c}ka}
\newline\indent
{\sl Max Planck Institute for Mathematics in Sciences, 
}
\newline\indent
{\sl Inselstrasse 22,}
\newline\indent
{\sl 041\,03~Leipzig}
\newline\indent
{\sl Germany}
\newline\indent
{\rm vodicka@mis.mpg.de}}

\begin{abstract}
The Kimura 3-parameter model is one of the most fundamental phylogenetic models in algebraic statistics. We prove that all algebraic varieties associated to this model are projectively normal, confirming a conjecture of Micha{\l}ek. 
\end{abstract}

\maketitle
\section{Introduction}
Phylogenetics is a science that models evolution. One of the central objects in phylogenetics is the \emph{tree model}. In general, a statistical model is a parametric family of probability distributions. The tree model is based on rooted tree and finite set $B$ and gives us probability distribution on $B^l$ where $l$ is the number of leaves of the tree. The parameters are distribution on the root and transition matrices along the edges of the tree. A group-based model is a tree model where the set $B$ is a group which acts on itself and parameters are $G$-invariant.

Since everything is finite a distribution allowed by a tree model may be represented as a vector $(p_1,\dots,p_n)$ where $p_i$'s are nonnegative and sum to one. Thus a tree model may be regarded as a map from the parameter space to the $n$-dimensional vector space.

In algebraic phylogenetics we are interested in the geometric locus of all probability distributions allowed by a given model. Precisely, the Zariski closure of this locus is an algebraic variety and one is interested in its geometric and algebraic properties \cite{4aut, SullivantAlgStat}. 

For example one asks for polynomials defining the variety---so-called phylogenetic invariants---or properties of the singular set. In this article, we investigate the latter property, namely we show that for a well-known 3-Kimura model \cite{kimura1981estimation}, the singularities are always \emph{normal}. This confirms a conjecture of Micha{\l}ek \cite[Conjecture 9.5]{JaJCTA}, \cite[Conjecture 12.1]{Diss}. 

The $3$-parameter Kimura model is a group-based model given by the group $\mathbb{Z}_2\times\mathbb{Z}_2$. Group-based models in general and the $3$-Kimura model in particular have been recently intensively studied within algebraic statistics \cite{SS, BW, mauhar2017h,DE, DK, JaJalg, JaAdvGeom, michalek2017phylogenetic, casanellas2008geometry, casanellas2015low, casanellas2017local, donten2016phylogenetic, michalek2017finite}.

Apart from the fact that it was an open conjecture, there are several important reasons to study normality of the $3$-Kimura model. 

\begin{itemize}
\item Group-based models allow a monomial parametrization \cite{SS}. Thus, one may say that they are \emph{toric varieties}. However, in pure mathematics one often requires a toric variety to be \emph{normal} \cite{Ful}. The reason is that in such a case the variety admits a nice combinatorial description in terms of a fan \cite{CLS}. Our result in particular implies that the normal fan of the polytope associated to the $3$-Kimura model describes the toric variety representing the model. 
\item Not all group-based models give rise to normal toric varieties: for example for the group $\mathbb{Z}_6$ one obtains a nonnormal variety \cite{DBM}. The normality also fails for the $2$-Kimura model. Thus the $3$-Kimura model is distinguished with respect to that regard.
\item Normality played an important role in the $\mathbb{Z}_2$ group-based model \cite{SX, BW, SS}.
\item Normality of toric varieties provides automatic bounds on degrees of phylogenetic invariants \cite{Stks}. In a special case of a tree with six leaves this was used in a recent proof \cite{michalek2017phylogenetic} of the Sturmfels-Sullivant conjecture \cite[Conjecture 30]{SS}. On that example normality was checked by computer using software Normaliz \cite{Normaliz}. Our proof, in particular, confirms normality in this case without the necessity to rely on computer software. 
\end{itemize}

It would not be possible to obtain our theorem without many great previous results. We list the most important below.
\begin{itemize}
\item \emph{Application of \textbf{D}iscrete \textbf{F}ourier \textbf{T}ransform to unravel toric structure} The DFT may be considered as a clever change of coordinates, that changes the parametrization of the phylogenetic model into one given by \emph{monomials}. First such applications were made by Handy and Penny \cite{hendy1989framework}. The toric structure was studied in detail in the work of Sturmfels and Sullivant \cite{SS} and Micha{\l}ek \cite{Diss}.
\item \emph{Reduction to claw trees} Recall that a claw tree is a tree with just one inner vertex.
It is known that one can extend many properties that hold for claw trees to arbitrary trees. This technique is well-developped to obtain phylogenetic invariants \cite{DK}. Further, it is known that \emph{normality} in case of claw trees implies normality for arbitrary trees. For phylogenetic group-based models it was first observed in \cite[Lemma 5.1]{JaJalg}. The joining of trees is a special case of a more general construction of toric fiber products \cite{Sethtfp, rauh2016lifting, engstrom2014multigraded}.
\item \emph{Facet description} The vertex description of the polytopes representing group-based models are well-known \cite{SS, JaJalg, BW}. However, obtaining facet description from the vertex one is hard in the general case, and for phylogenetic models in particular. For the 3-Kimura model such a description was provided in \cite{mauhar2017h}.
\end{itemize}
First two results of the above allow us to translate the question about projective normality of the variety associated to the 3-Kimura model into a purely combinatorial statement about normality of a family of polytopes. We prove the normality using only combinatorial methods. Strong tool is the facet description of the polytope because it allows us to prove that a point lies inside of the polytope by checking inequalities.

\section{The polytope of the 3-Kimura model}
We start by fixing notation.

Let $G$ be the group $\mathbb Z_2\times \mathbb{Z}_2$. Let us denote its elements by $0,\alpha,\beta,\gamma$. We also denote the elements of $\mathbb Z_2$ by $\mathbf{0},\mathbf{1}$.

Let $H_n$ be the set of the group-based flows of length $n$ of $G$, i.e. $$H_n=\{(g_1,\dots,g_n)\in G:g_1+g_2+\dots+g_n=0\}.$$
It is easy to see that $H_n$ is a subgroup of $G^n$.

The goal of this article is to prove normality of a family of polytopes for 3-Kimura model $P_n\subset\mathbb R^{4n}$ indexed by $n\in\mathbb N$. Before we formally define them, we introduce further notation.

We denote the coordinates of a point $x\in \mathbb R^{4n}$ by $x_g^j$ where $g\in G$ and $1\le j\le n$. Although we are using upper indices, there will be no ambiguity since we will not use any powers in this article.

\begin{defin}
We say that \textit{the $G$-presentation} of a point $x\in \mathbb Z_{\ge 0}^{4n}$ is an $n$-tuple $(G_1,\dots,G_n)$ of multisets of elements of $G$ such that the element $g\in G$ appears exactly $x_g^j$ times in the multiset $G_j$. We may identify the $n$-tuple $(g_1,\dots,g_n)\in G^n$ with the $n$-tuple of multisets $(\{g_1\},\dots,\{g_n\})$.
\end{defin}

\begin{defin} The vertices of $P_n$ are all points of $\mathbb R^{4n}$ which $G$-presentations are the $n$-tuples from $H_n$. Therefore, $P_n$ is a convex hull of these points. 
\end{defin}

Equivalent characterization of $P_n$ is given in \cite{mauhar2017h}. The polytope is defined by the following inequalities: 
\begin{itemize}
\item $x_g^j\ge 0$ for all $g\in G,1\le j\le n$,
\item $x_0^j+x_\alpha^j+x_\beta^j+x_\gamma^j=1, \text{ for all } 1\le j\le n$,
\item For all $A\subseteq\{1,2,\dots,n\}$ with $|A|$ being an odd number:
$$\sum_{j\in A} (x_0^j+x_\alpha^j)+\sum_{j\not\in A} (x_\beta^j+x_\gamma^j)\ge 1,$$
$$\sum_{j\in A} (x_0^j+x_\beta^j)+\sum_{j\not\in A} (x_\alpha^j+x_\gamma^j)\ge 1,$$
$$\sum_{j\in A} (x_0^j+x_\gamma^j)+\sum_{j\not\in A} (x_\alpha^j+x_\beta^j)\ge 1.$$
\end{itemize}

We denote the left sides of the last three inequalities by $S_\alpha(x,A), S_\beta(x,A), S_\gamma(x,A)$ respectively. Each inequality gives us a facet of $P_n$. We define $$F_g(A)=\{x\in P_n: S_g(x,A)=1\}.$$ 

The lattice generated by vertices of $P_n$ is $$L_n=\{m\in \mathbb Z^{4n}|\ \forall 1\le j,j'\le n: \sum_{g\in G} m_g^j=\sum_{g\in G} m_g^{j'}, \sum_{\substack {g\in G\\1\le j\le n}} (m_g^jg)=0\},$$

where the last sum is in $G$. Alternatively, we can characterize $G$-presentations of points in $L_n\cap \mathbb Z_{\ge 0}^n$ as follows: Every multiset has the same size and sum of all elements in multisets is $0$.

\begin{defin} Let $v(0)$ be the vertex corresponding to the $n$-tuple $(0,\dots,0)$ and $v(g)_{j,j'}$ be the vertex corresponding to the $n$-tuple which has on $j$-th and $j'$-th place $g$ and all other places 0.

Let $V_n$ be the following set of vertices of $P_n$:
$$V_n=\{v(0)\}\cup\{v(g)_{j,n}|1\le j\le n-1,g\in\{\alpha,\beta,\gamma\}\}.$$
\end{defin}

\vspace{0.2 cm}
Our goal is to prove that $P_n$ is normal for every positive integer $n$. Let us recall that polytope $P_n$ is \emph{normal} if every point in $kP_n\cap L_n$ can be written as a sum of $k$ lattice points from $P_n$. Normality of polytope is equivalent to the fact that the associated projective toric variety is projectively normal.

It is easy to check that $P_1$, $P_2$ and $P_3$ are normal. Hence, in this article we consider only $n\ge 4$.

\section{symmetries of $P_n$}

Polytope $P_n$ has a lot of symmetries that can be described by group actions on $\mathbb R^{4n}$:

\begin{itemize}
\item Action of $\mathbb S_n$:

For $\sigma\in \mathbb S_n$ and $x\in\mathbb R^{4n}$ we define $\sigma(x)_g^{\sigma(j)}=x_g^j$. Intuitively, we only permute quadruples of coordinates by the upper index.

\item Action of $H_n$: 

For $h=(g_1,\dots,g_n)\in H_n$ and $x\in\mathbb R^{4n}$ and we define $(hx)_g^j=x_{(g+g_j)}^j$.
Intuitively, if we look at $G$-presentation of a point in $\mathbb Z^{4n}$ we add $g_j$ to elements in $G_j$.

\item Action of $\Aut (G)$:

For $\varphi\in \Aut (G)$ and $x\in\mathbb R^{4n}$ we define $\varphi(x)_{\varphi(g)}^j=x_g^j$. Again, if we consider $G$-presentation of $x$ this is application of the automorphism $\varphi$ to elements in multisets.
\end{itemize} 

All of these actions only permute coordinates in $\mathbb R^{4n}$ and therefore are automorphisms of $\mathbb R^{4n}$ as a vector space. It can be easily verified that they map vertices of $P_n$ to vertices of $P_n$ and therefore preserve $P_n$. It follows that these actions restricted to $L_n$ are automorphisms of this lattice.

We want to prove that every point $x\in kP_n\cap L$ decomposes to a sum of $k$ lattice points from $P_n$. It is enough to prove it for an image of $x$ under any of group actions described above, since $Ax=v_1+\dots+ v_k$ implies $x=A^{-1}v_1+\dots+A^{-1}v_k$ for any group action $A$.

\vspace{0.3cm}
Let us define linear ordering on multisets of four real numbers with sum equal to $k$. Consider two multisets $\{a,b,c,d\}$ and $\{a',b',c',d'\}$. Without loss of generality we may assume $a\ge b\ge c\ge d$ and $a'\ge b'\ge c'\ge d'$. We say $$\{a,b,c,d\}\succ \{a',b',c',d'\}\Leftrightarrow (a>a') \vee (a=a'\wedge b>b')\vee (a=a'\wedge b=b'\wedge c>c').$$

Consider $x\in kP\cap L$. If we order multisets $\{x_0^j,x_\alpha^j,x_\beta^j,x_\gamma^j\}$ then by acting with corresponding permutation from $\mathbb S_n$ we can ensure that multiset for $j=n$ is the smallest one in this ordering.

Let us denote $g_j$ the most frequent element (or one of the most frequent elements) in $j$-th multiset from $G$-presentation of $x$, i.e. $x_{g_j}^j=\max\{x_0^j,x_\alpha^j,x_\beta^j,x_\gamma^j\}$. Then by acting with $(g_1,g_2,\dots,g_{n-1},g_1+\dots+g_{n-1})\in H_n$ we obtain a point $x$ in which the element $0$ is the most frequent in all multisets except the last one.

This means that if we need to, for a point $x\in kP_n\cap L$ we may without loss of generality assume the following two facts:

\begin{equation}
\forall j\in\{1,2,\dots,n-1\}:\ \{x_0^j,x_\alpha^j,x_\beta^j,x_\gamma^j\}\succeq\{x_0^n,x_\alpha^n,x_\beta^n,x_\gamma^n\}
\end{equation}
\begin{equation}
\forall j\in\{1,2,\dots,n-1\}:\ x_0^j=\max\{x_0^j,x_\alpha^j,x_\beta^j,x_\gamma^j\}
\end{equation}

We add another definition:

\begin{defin}
Let $x\in kP_n\cap L_n$. The vertex $v$ of $P_n$ is called \emph{$x$-good} if all coordinates of the point $x-v$ are non-negative. 
\end{defin}

\section{Preliminary results}
\begin{lem}\label{1}
Let $x\in kP_n\cap L_n$ and $j\in\{1,2,\dots,n\}$. Suppose that $x_0^j\ge x_\alpha^j,x_\beta^j,x_\gamma^j$ and let $g\in\{\alpha,\beta,\gamma\}$. Then $x_0^j+x_g^j\ge \lceil k/3\rceil$ and the equality holds if and only if $x_g^j=0$ and $x_h^j=k/3$ for $h\neq g$.
\end{lem}
\begin{proof}
$$3(x_0^j+x_g^j)\ge 3x_0^j+x_g^j\ge x_0^j+x_\alpha^j+x_\beta^j+x_\gamma^j=k.$$
We divide by 3 and realise that $x_g^j$ are integers to obtain wanted inequality. The part about the equality is obvious.
\end{proof}

\begin{lem}\label{2}
Let $x\in L_n$ be a point such that $x_0^j+x_\alpha^j+x_\beta^j+x_\gamma^j=k$ for all $1\le j\le n$. Let $g\in\{\alpha,\beta,\gamma\}$ and $A\subseteq\{1,2,\dots,n\}$ be a set of odd cardinality. Then $$S_g(x,A)\equiv k\pmod 2.$$
\end{lem}
\begin{proof}
We consider only the case $g=\alpha$, other cases are analogous.

Consider the homomorphism $$\varphi: \mathbb Z_2\times \mathbb Z_2\rightarrow \mathbb Z_2$$ $$0,\alpha\mapsto \mathbf 0,\ \beta,\gamma\mapsto \mathbf 1.$$

For $x\in L_n$ we get $$\mathbf 0=\varphi(0)=\varphi\left(\sum_{\substack {g\in G\\1\le j\le n}} (x_g^jg)\right)=\sum_{\substack {g\in G\\1\le j\le n}}(x_g^j\varphi(g))=\sum_{j=1}^n(x_\beta^j+x_\gamma^j)\cdot \mathbf 1.$$
Therefore $\sum_{j=1}^n(x_\beta^j+x_\gamma^j)$ must be even. This implies
$$\sum_{j\in A} (x_0^j+x_\alpha^j)+\sum_{j\not\in A} (x_\beta^j+x_\gamma^j)=\sum_{j=1}^n(x_\beta^j+x_\gamma^j)+\sum_{j\in A}(x_0^j+x_\alpha^j-x_\beta^j-x_\gamma^j)=$$
$$=\sum_{j=1}^n(x_\beta^j+x_\gamma^j)-\sum_{j\in A}(x_0^j+x_\alpha^j+x_\beta^j+x_\gamma^j)+2\sum_{j\in A}(x_0^j+x_\alpha^j)=$$
$$\sum_{j=1}^n(x_\beta^j+x_\gamma^j)-k|A|+2\sum_{j\in A}(x_0^j+x_\alpha^j)\equiv k|A|\equiv k\pmod 2.$$ 
\end{proof}

The following lemma implies that it is sufficient to consider only such points $x$ for which the following condition holds:

\begin{equation}
\forall j\in\{1,2,\dots,n\},g\in G: x_g^j<k
\end{equation} 

\begin{lem}\label{3}
Suppose that for every positive integers $k,m$ and every $x\in kP_m\cap L_m$ such that $x_g^j<k$ for all $g,j$ we can write $x$ as a sum of $k$ vertices of $P_m$. Then $P_n$ is normal for every positive integer $n$.
\end{lem}
\begin{proof}
Proof by induction on $n$. $P_1,\ P_2$ and $P_3$ are normal.

Suppose that $P_{n-1}$ is normal. We prove that also $P_n$ is normal. Consider a point $x\in kP_n\cap L_n$. If $x_g^j<k$ for all $g,j$ then $x$ decomposes by assumption. Therefore we may assume that $x_g^j=k$ for some $g,j$. By acting with suitable permutation we can assume that $j=n$ and then by acting with $(g,0,\dots,0,g)$ we obtain $g=0$.

Consider now the projection $\pi:\mathbb R^{4n}\rightarrow \mathbb R^{4(n-1)}$ on the first $4(n-1)$ coordinates. Since $x\in kP_n$ there exist positive real numbers $\lambda_1,\dots,\lambda_s$ with $\lambda_1+\dots+\lambda_s=k$ such that $\lambda_1v_1+\dots+\lambda_sv_s=x$, where $v_1,\dots,v_s$ are some vertices of $P_n$. But $(v_i)_0^n\le 1$ and $x_0^n=k$ implies $(v_i)_0^n=1$ for all $i$.

Consequently, $\pi(v_i)$ is a vertex of $P_{n-1}$ and $\pi(x)\in kP_{n-1}$. By induction hypothesis $\pi(x)$ decomposes to $\pi(x)=u_1+\dots+u_k$, where $u_i$ are vertices of $P_{n-1}$.

Now we simply put $u_i'\in \pi^{-1}(u_i)$ such that $(u_i')_0^n=1$ and $(u_i')_g^n=0$ for $g\neq 0$.
Obviously all $u_i'$ are vertices of $P_n$ and we have $x=u_1'+\dots+u_n'$.
\end{proof}

\begin{lem}\label{4}
Let $x\in 2P_n\cap L_n$. Then $x$ can be written in the form $x=v+v'$ where $v,v'$ are vertices of $P_n$.
\end{lem}
\begin{proof}
We assume, without loss of generality, $(1)$, $(2)$ and $(3)$. Thus $x_0^j>0$ for $1\le j\le n-1$. If also $x_0^n> 0$ then $x=v(0)+(x-v(0))$. Further, $x-v(0)$ must be a vertex of $P_n$ since it has non-negative coordinates and sum of elements in $G$-presentation of $x-v(0)$ is $0$ since it is $0$ for both $x$ and $v_0$. If $x_0^n=0$ then by acting with suitable $\varphi\in \Aut (G)$ we have $x_\alpha^n=x_\beta^n=1$ since $x_g^n<2$ for all $g$ by condition $(3)$.

Since $S_\gamma(x,\{n\})\ge 2$ at least one of the numbers $x_g^j$ for $g=\alpha,\beta$; $1\le j\le n-1$ is greater than 0. Then $x=v(g)_{j,n}+(x-v(g)_{j,n})$ for such $g,j$. By the same arguments as above $(x-v(g)_{j,n})$ must be a vertex of $P_n$. 
\end{proof}

\begin{lem}\label{6}
Let $x\in kP_n\cap L_n$ be such that there are at least three multisets $\{k/3,k/3,k/3,0\}$ in $G$-presentation of $x$. Then $x=y+v$, where $v$ is a vertex of $P_n$ and $y\in (k-1)P_n$.
\end{lem}
\begin{proof}
By acting with suitable permutation from $\mathbb S_n$ we may assume that these three multisets are the first three. Then by acting with suitable $(g_1,g_2,\dots,g_n)\in H_n$ we may assume $x_0^j=x_\alpha^j=x_\beta^j=k/3$ for $j=1,2,3$. We describe the $G$-presentation of $v$ (which is a $n$-tuple $(g_1,g_2,\dots,g_n)$ of elements from $G$). We pick the last $n-2$ elements arbitrarily, the only condition is that $g_j$ belongs to the $j$-th multiset from $G$-presentation of $x$. Then we pick $g_1$ and $g_2$ such that sum of this $n$-tuple is 0 and $g_1,g_2,g_3$ are not all equal. Since $g_1$ and $g_2$ can be any from $0,\alpha,\beta$, it is possible.  

Now we need to check that $x-v=y\in (k-1)P_n$. We only need to check the inequalities for sets $A$. However, if we try to compute $S_g(y,A)$ we always get at least $k-2$ already on the first three coordinates. Therefore, due to Lemma \ref{2} the inequalities hold.
\end{proof}

From now, we may assume that $x\in kP_n\cap L_n$ satisfies the following condition since the other case is covered by the previous lemma.

\begin{equation}
\text{At most two multisets from } G \text{-presentation of }x \text{ are }\{k/3,k/3,k/3,0\}.
\end{equation}

\begin{lem}\label{7}
Let $x\in 3P_n\cap L_n$ satisfy $(2)$, $(4)$. Let $A\subseteq\{1,2,\dots,n\}$ be a set with $|A|\ge 5$. Then $S_g(x-v,A)\ge 2$ for any $g\in\{\alpha,\beta,\gamma\}$ and any $x$-good vertex $v$ of $P_n$.
\end{lem}

\begin{proof}
Let  $B\subseteq \{1,2,\dots,n\}$ be the set of those indices for which multisets in $G$-presentation of $x$ are equal to $\{1,1,1,0\}$. This together with condition $(2)$ yields $x_0^j\ge 2$ for $j\not\in B$. Condition $(4)$ implies $|B|\le 2$. It follows that $$S(x-v,A)\ge \sum_{j\in A\setminus B} (x_0^j-1)\ge |A\setminus B|\ge 2.$$
\end{proof}

\section{The proof}
\subsection{Idea of the proof}
We prove for all positive integers $k,n$ that every point $x\in kP_n\cap L_n$ can be written in the form $x=y+v$ where $y\in (k-1)P_n$ and $v$ is a vertex of $P_n$. This, of course, means that also $y\in L_n$ since all vertices of $P_n$ belong to $L_n$ and this implies that $P_n$ is normal.

Consider a point $x\in kP_n\cap L_n$. It is sufficient to consider $k\ge 3$ because the case $k=2$ is solved by Lemma \ref{4}. Without loss of generality, from now we will suppose that $x$ satisfies $(1)$, $(2)$, $(3)$ and $(4)$. To conclude we need to pick an $x$-good vertex $v$ and then check that $y=x-v$ belongs to $(k-1)P_n$. We prove this by checking all inequalities from facet characterization of $P_n$ for every set $A$ with odd cardinality. 
 
Regarding the vertex $v$, we show we can always use some vertex $v\in V_n$ as in Definition 3. 

\subsection{Big sets $A$}
\begin{prop}\label{big}
Let $x\in kP_n\cap L_n, k\ge 3$ satisfy $(1)-(4)$ and let $A\subseteq\{1,2,\dots,n\}$ be a set with odd cardinality.  \begin{itemize}
\item[a)]If $|A|\ge 5$ then $S_g(x-v,A)\ge k-1$ for any $g\in\{\alpha,\beta,\gamma\}$ and any $x$-good vertex $v$ of $P_n$.
\item[b)] If $|A|=3,n\not\in A$ and $x$ satisfies $(4)$ then $S_g(x-v,A)\ge k-1$ for any $g\in\{\alpha,\beta,\gamma\}$ and any $x$-good vertex $v$ of $P_n$.
\end{itemize}
\end{prop}

\begin{proof}
Let $y=x-v$. Clearly, it is sufficient to prove the inequality for $g=\alpha$. We begin with part a):
\begin{align*}
\sum_{j\in A} (y_0^j+y_\alpha^j)+\sum_{j\not\in A} (y_\beta^j+y_\gamma^j)&\ge \sum_{j\in A\setminus\{n\}} (y_0^j+y_\alpha^j)\\
&\ge \sum_{j\in A\setminus\{n\}} (x_0^j+x_\alpha^j-1)  \\
&\ge \sum_{j\in A\setminus\{n\}} (\lceil k/3\rceil -1) \\
&\ge 4\lceil k/3\rceil-4\ge k-2.
\end{align*}

The last inequality holds for $k\ge 4$. Case $k=3$ is covered in Lemma \ref{7}. We also used Lemma \ref{1} and $|A\setminus\{n\}|\ge 4$. Inequality $S_g(y,A)\ge k-2$ together with Lemma \ref{2} implies $S_g(y,A)\ge k-1$.

Proof of part b) is similar:

\begin{align*}
\sum_{j\in A} (y_0^j+y_\alpha^j)+\sum_{j\not\in A} (y_\beta^j+y_\gamma^j)&\ge \sum_{j\in A} (y_0^j+y_\alpha^j)\\
&\ge \sum_{j\in A} (x_0^j+x_\alpha^j-1)  \\
&\ge 3\lceil k/3\rceil-3\ge k-3,
\end{align*}

where we again used Lemma \ref{1}. Lemma \ref{2} implies that $S_g(y,A)\neq k-2$. Therefore the only bad case is when we have an equality. This is possible only if we have equality everywhere, in particular $x_0^j=x_\beta^j=x_\gamma^j=k/3$ for all $j\in A$. But this means that $x$ does not satisfy $(4)$ which is a contradiction.
\end{proof}
Therefore it is sufficient to check inequalities for $|A|=1$ and $|A|=3$ such that $n\in A$.

\subsection{Small sets $A$}
Since $x\in kP_n$ we have the inequalities $S_g(x,A)\ge k$ for any $g$ and any set $A$ with odd cardinality. For big sets $A$ discussed in the section 5.2 we have not used them. However, we use them for smaller sets. Our first step is to observe how does $S_g(x,A)$ change when we subtract some vertex $v\in V_n$ from $x$. 
\begin{lem}\label{small}
Let $x\in kP_n\cap L_n$, $v\in V_n$, $g\in\{\alpha,\beta,\gamma\}$ and $|A|=1$ or $|A|=3$ with $n\in A$. Then $$S_g(x-v,A)=S_g(x,A)-3 \text{ or } S_g(x-v,A)=S_g(x,A)-1.$$ Moreover, for $|A|=1$ we have $S_g(x-v,A)=S_g(x,A)-1$ if and only if one of the following conditions holds:\begin{itemize}
\item $v=v(0)$
\item $v=v(g)_{j,n}$ for any $1\le j\le n-1$
\item $v=v(g')_{j,n}$ for $g'\neq g$ and $A=\{j\}$ or $|A|=\{n\}$
\end{itemize}
Also $S_g(x-v,\{n\})\ge k-1$.
\end{lem}
\begin{proof}
For the first part, one checks how many summands in $S_g(x,A)$ will decrease by 1 when we subtract $v$. The last part is clear consequence since $S_g(x,\{n\})\ge k$ for $x\in kP_n$.
\end{proof}

Now we consider the following:

\begin{prop}\label{0} Let $x\in kP_n\cap L_n$ satisfy conditions $(1)-(4)$. Suppose that $0$ is also the most frequent element in the $n$-th multiset from $G$-presentation of $x$. Then $x-v(0)\in (k-1)P_n$.
\end{prop}
\begin{proof}
Obviously, every multiset from $G$-presentation of $x$ contains $0$ so $x-v(0)$ has non-negative coordinates and therefore $v(0)$ is $x$-good. Inequalities for sets with $|A|\ge 3$ hold for $x-v$ by Proposition \ref{big} since for sets with $|A|=3$ and $n\in A$ we can use same arguments. Inequalities for $|A|=1$ hold by Lemma \ref{small} since we are subtracting $v(0)$.
It follows that $x-v(0)\in (k-1)P_n$. 
\end{proof}

The previous proposition implies that we can assume that for $x\in kP_n\cap L_n$ satisfying $(1),(2),(3)$ also the following condition holds:
\begin{equation}
\begin{aligned}
&\text{There exists no } h\in H_n\text{ such that the following conditions holds:} 
\\ &0\text{ is the most frequent element in all multisests from } G \text{-presentation of }hx.
\end{aligned}
\end{equation}

\begin{prop}\label{A3}
Let $x\in kP_n\cap L_n$ satisfy $(1)-(5)$. Then:
\begin{itemize} 
\item[a)]$x$ does not belong to any facet $kF_g(A)$ for $|A|=3,n\in A$, i.e. $S_g(x,A)>k$ for all such $A$ and $g=\alpha,\beta,\gamma$.
\item[b)] $S_g(x-v,A)\ge k-1$ for all $v\in V_n$, $g\in\{\alpha,\beta,\gamma\}$ and $|A|=3,n\in A$.
\end{itemize}
\end{prop}
\begin{proof}
We prove part a) by contradiction: Suppose that we have an equality for $A=\{1,2,n\}$ and $g=\alpha$. We may get to this situation by acting with suitable $\sigma\in \mathbb S_n$ and $\varphi\in \Aut (G)$. We compute $S_\alpha(x,A)$:

$$S_\alpha(x,A)=\sum_{j\in A} (x_0^j+x_\alpha^j)+\sum_{j\not\in A} (x_\beta^j+x_\gamma^j)\ge x_0^1+x_0^2+x_0^n+x_\alpha^n=$$
$$=x_0^1+x_0^2+k-(x_\beta^n+x_\gamma^n)\ge x_0^1+x_0^2+k-2\min\{x_0^1,x_0^2\}\ge k.$$

An equality holds only if there is equality in all inequalities. In particular, it means that $x_0^1=x_0^2=x_\beta^n=x_\gamma^n$ and $x_\alpha^1=0$. But from ordering of multisets, we get that also some $x_{g_0}^1=x_0^1$ for $g_0=\beta$ or $g_0=\gamma$. By acting with $(g_0,0,\dots,0,g_0)\in H$ we get to the situation where $0$ is the most frequent also in $n$-th multiset and still is also most frequent on the first one. This is a contradiction with condition $(5)$.

We continue with proof of part b). Part a) together with Lemma \ref{2} implies that $S_g(x,A)\ge k+2$. Consequently, Lemma \ref{small} implies $S_g(x-v,A)\ge k-1$ for any $v\in V_n$.
\end{proof}

\begin{prop}\label{x_n>0}
Let $x\in kP_n\cap L_n$ satisfy $(1)-(5)$ and $x_0^n>0$. Then $x-v(0)\in (k-1)P_n$.
\end{prop}
\begin{proof}
Clearly, $v(0)$ is $x$-good. Inequalities for $|A|\ge 3$ hold by Propositions \ref{big} and \ref{A3}. For $|A|=1$ we have $S_g(x,A)\ge k$, then by Lemma \ref{small} we get $S_g(x-v(0),A)\ge k-1$. Since all inequalities hold $x-v(0)\in (k-1)P_n$.
\end{proof}

Therefore we are left only with the case $x_0^n=0$.

\subsection{Special case $x_0^n=0$}
In this case we will subtract a vertex $v(g)_{j,n}$ for a special choice of $g$ and $j$. Propositions \ref{big} and \ref{A3} and Lemma \ref{small} imply that it is enough to check inequalities for $|A|=1, A\neq\{n\}$.
We distinguish two cases depending on whether $x$ lies or does not lie on a facet $kF_g(A)$ for such $A$.

\begin{prop}\label{0nof}
Let $x\in kP_n\cap L_n$ satisfy $(1)-(5)$, $x_0^n=0$ and $x$ does not belong to any facet $kF_g(A)$ for $|A|=1,A\neq\{n\}$. Then there exists a vertex $v\in V_n$ such that $x-v\in (k-1)P_n$.
\end{prop}

\begin{proof}
For any $v\in V_n $ Lemma \ref{small} implies that for any $A$ with $|A|=1,A\neq\{n\}$ we have $S_g(x-v,A)\ge S_g(x,A)-3\ge k-1$. We used Lemma \ref{2} to deduce inequality $S_g(x,A)\ge k+2$. Therefore, inequalities for every set $A$ hold for any $x$-good vertex $v\in V_n$, since bigger sets are taken care of by Propositions \ref{big} and \ref{A3}. Consequently, it is sufficient to pick any $x$-good vertex $v\in V_n$. 
 
At least two of the numbers $x_g^n$ for $g\in\{\alpha,\beta,\gamma\}$ must be non-zero by condition $(3)$ and the fact that $x_0^n=0$. Without loss of generality, let those two coordinates be $x_\alpha^n$ and $x_\beta^n$.

Since $S_\gamma(x,\{n\})\ge k$ and $x_0^n+x_\gamma^n<k$, at least one of the numbers $x_\alpha^j$ and $x_\beta^j$ for $1\le j\le n-1$ must be non-zero. Let it be $x_{g_0}^j$. For $v=v(g_0)_{j,n}$ all coordinates of $x-v$ are non-negative since condition $(2)$ implies $x_0^j>0$ for $1\le j\le n-1$. This means we have found $x$-good vertex $v\in V_n$ and the proposition is proved.
\end{proof}

If $x$ belongs to a facet we prove that it belongs to only one facet and that we can as well subtract a vertex $v\in V_n$:

\begin{prop}\label{0f}
Let $x\in kP_n\cap L_n$ satisfy $(1)-(4)$, $x_0^n=0$ and $x$ belongs to some facet $kF_g(A)$ for $|A|=1,A\neq\{n\}$. Then
\begin{itemize}
\item[a)] $x$ belongs to only one such facet.
\item[b)] There exists a vertex $v\in V_n$ such that $x-v\in (k-1)P_n$.
\end{itemize}
\end{prop}
\begin{proof}

By acting with suitable permutation from $\mathbb S_n$ and $\varphi\in \Aut (G)$ we can get to situation where $x\in kF_\alpha(\{1\})$. We have

 $$k=S_\alpha(x,\{1\})\ge x_0^1+x_\beta^n+x_\gamma^n=x_0^1+k-x_\alpha^n\ge k.$$

To get an equality, there must be an equality in all inequalities, specifically $x_\alpha^1=x_\beta^j=x_\gamma^j=0$ for all $2\le j\le n$ and $x_0^1=x_\alpha^n=\max_{g\in G}\{x_g^n\}$. 

Assumption that $x$ belongs to a facet give us strong conditions. It is easy to see that $x$ cannot belong to some other facet $kF_\alpha(\{j\})$ for $j<n$ because it would imply $x_\beta^1=x_\gamma^1=0$. But this is a contradiction with condition $(3)$. Also $x$ cannot belong to some $kF_\beta(\{j\})$ for $1\le j<n$ because it would imply $x_\alpha^i=x_\beta^i=x_\gamma^i$ for $i\neq 1,j,n$ which is again a contradiction with $(3)$. Same arguments hold for $kF_\gamma(\{j\})$. This proves part a).

For part b), by the same arguments as in the proof of Proposition \ref{0nof} for any $x$-good vertex $v\in V_n,\ g\in\{\alpha,\beta,\gamma\}$ and set $A$ we have $S_g(x-v,A)\ge k-1$, except the case when $g=\alpha$ and $A=\{1\}$. 

Since $k\le S_\alpha(x,\{n\})=x_\beta^1+x_\gamma^1+x_\alpha^n$ at least one of the numbers $x_\beta^1$, $x_\gamma^1$ must be greater than 0. Also one of the numbers $x_\beta^n$ and $x_\gamma^n$ is greater than zero by condition $(3)$.

If both numbers $x_g^n$ and $x_g^1$ are greater than zero for $g=\beta$ or $g=\gamma$ then the vertex $v=v(g)_{1,n}$ is $x$-good. By Lemma \ref{small} also $S_\alpha((x-v),\{1\})\ge k-1$  and therefore $x-v\in (k-1)P_n$.

 Suppose the opposite, i.e. $x_\beta^1=0$ and $x_\gamma^n=0$ (we can get to this case by acting with $\varphi\in \Aut (G)$). Then $S_\gamma(x,\{n\})\ge k$ implies that at least one of the numbers $x_\alpha^j$ for $2\le j\le n-1$ is greater than 0. Then we can subtract $v=v(\alpha)_{j,n}$ for such $j$.
 Again $x-v$ has non-negative coordinates and by Lemma \ref{small} $S_\alpha((x-v),\{1\})\ge k-1$.

\end{proof}

\begin{thm}
Polytope representing 3-Kimura model $P_n$ is normal for every positive integer $n$.
\end{thm}

\begin{proof}
Consider point $x\in kP_n\cap L_n$ for some positive integer $k$. If $k=2$ then $x$ decomposes due to Lemma \ref{4}. To prove normality of $P_n$ it is sufficient for $k\ge 3$ to prove that there exists a vertex $v$ of $P_n$ such that $x-v\in (k-1)P_n$. Also it is sufficient to consider only points $x$ which satisfy $(1)-(3)$. The existence of such $v$ is implied by Lemma \ref{6} and Propositions \ref{0}, \ref{x_n>0}, \ref{0nof} and \ref{0f}.
\end{proof}

\printbibliography
\end{document}